\documentclass[12pt,reqno]{amsart}

\newcommand\version{March 23, 2015}

\usepackage{amsmath,amsfonts,amsthm,amssymb,amsxtra}

\usepackage[ansinew]{inputenc}

\usepackage{textcomp}

\usepackage{cite}

\newtheorem{theorem}{Theorem}
\newtheorem{proposition}[theorem]{Proposition}
\newtheorem{lemma}[theorem]{Lemma}

\theoremstyle{definition}

\theoremstyle{remark}

\renewcommand{\epsilon}{\varepsilon}

\newcommand{\loc}{{\rm loc}}

\newcommand{\N}{\mathbb{N}}

\renewcommand{\phi}{\varphi}
\newcommand{\rad}{{\rm rad}}
\newcommand{\R}{\mathbb{R}}

\DeclareMathOperator{\tr}{Tr}

\begin{document}

\title[Sharp eigenvalue bounds --- \version]{Sharp eigenvalue bounds on quantum star graphs}

\author{Semra Demirel--Frank}
\address{Semra Demirel--Frank, Mathematics 253-37, Caltech, Pasadena, CA 91125, USA}
\email{sdemirel@caltech.edu}

\begin{abstract}
We prove that the optimal constant in the Lieb--Thirring inequality on a star graph with $N$ edges coincides with that on $\R$ if $N$ is even. For odd $N$ we show that this property holds when restricting to radial potentials and we prove an almost optimal bound for general potentials.
\end{abstract}

\maketitle

\section{Introduction}

Recently there has been a lot of activity in a mathematical understanding of quantum graphs, which appear as idealized models of linear, network-shaped structures in mesoscopic physics. A large literature on the subject has arisen and we refer, for instance, to the bibliography given in \cite{BK, EKKST} and the textbook \cite{BeKu}. In particular, in the papers \cite{EFK,DH,FK,ExLaUs} bounds we derived on the discrete eigenvalues of Schr\"o\-din\-ger operators on metric graphs. In the present paper we will be interested in \emph{optimal constants} in such bounds for one of the simplest classes of metric graphs, namely \emph{star graphs}. By $\Gamma_N$ we denote $N$ half-lines $[0,\infty)$ with their endpoints $0$ identified. Thus, $\Gamma_N$ is a graph with a single vertex and $N$ edges.

We consider the Schr\"odinger operator
$$
H= -\frac{d^2}{dx^2} + V
\qquad\text{in}\ L_2(\Gamma_N)
$$
with a potential $V:\Gamma_N\to\R$. It is well-known that, if $V_-\in L_p(\Gamma_N)$ for some $p\geq 1$ and $V_+\in L_1^\loc(\Gamma_N)$, then the Schr\"odinger operator can be defined as a self-adjoint operator in $L_2(\Gamma_N)$ via the lower semi-bounded and closed quadratic form
$$
h[\psi]:=\int_{\Gamma_N} \left(|\psi'|^2 + V|\psi|^2\right) dx \,,
\qquad \psi\in H^1(\Gamma_N)\cap L_2(\Gamma_N,V_+\,dx) \,.
$$
By definition, a function $\psi$ on $\Gamma_N$ belongs to the Sobolev space $H^1(\Gamma_N)$ if its $N$ restrictions $\psi_1,\ldots,\psi_N$ to the edges of $\Gamma_N$ belong to $H^1(0,\infty)$ and if their values at the vertex coincide. This definition of the Schr\"odinger operator via quadratic forms gives rise, in a generalized sense, to the so-called \emph{Kirchhoff boundary conditions} at the vertex,
$$
\sum_{j=1}^N \psi_j'(0+) = 0 \,.
$$
Moreover, the condition $V_-\in L_p(\Gamma_N)$ with $p<\infty$ guarantees that the negative spectrum of the Schr\"odinger operator consists of discrete eigenvalues of finite multiplicities. As usual, we write $\tr H_-^\gamma$ for the sum of the $\gamma$-th power of the absolute values of the negative eigenvalues of $H$.

One can prove \cite{EFK} that for any $\gamma\geq 1/2$ there is a constant $L_{\gamma,N}$ such that
\begin{equation}
\label{eq:ltstar}
\tr H_-^\gamma \leq L_{\gamma,N} \int_{\Gamma_N} V_-^{\gamma+1/2}\,dx \,.
\end{equation}
In the following, we will denote by $L_{\gamma,N}$ the \emph{optimal} (that is, smallest possible) value of the constant in \eqref{eq:ltstar}. We are interested in characterizing this value and, in particular, in relating it to $L_{\gamma,2}=:L_\gamma$ for $\Gamma_2=\R$, that is, the optimal constant in the inequality
\begin{equation}
\label{eq:lt}
\tr \left( -\frac{d^2}{dx^2} + V \right)_-^\gamma \leq L_{\gamma} \int_{\R} V_-^{\gamma+1/2}\,dx \,.
\end{equation}
Finding the optimal constant in \eqref{eq:lt} is a famous open problem due to Lieb and Thirring \cite{LT}. What is currently known is that
\begin{equation}
\label{eq:ltknown}
L_{1/2}= 1/4 
\qquad\text{and}\qquad
L_\gamma = (4\pi)^{-1/2} \frac{\Gamma(\gamma+1)}{\Gamma(\gamma+3/2)}
\ \text{if}\ \gamma\geq 3/2 \,;
\end{equation}
see \cite{HLThom, LT} and also \cite{LW,H} for a review and results in higher dimensions.

By taking a compactly supported almost-optimal potential for \eqref{eq:lt} and transplanting it very far out on a single edge of $\Gamma_N$ it is easy to see that
\begin{equation}
\label{eq:comparision0}
L_{\gamma,N} \geq L_\gamma
\qquad\text{for all}\ \gamma\geq 1/2 \ \text{and all}\ N\in\N \,.
\end{equation}
Thus, in the following we will be interested in \emph{upper bounds} on $L_{\gamma,N}$.

In \cite{DH} we have shown that
\begin{equation}
\label{eq:dh}
L_{\gamma,N} = L_\gamma
\qquad\text{for all}\ \gamma\geq 2 \ \text{and all}\ N\in\N \,.
\end{equation}
In fact, this equality is valid for a large number of graphs, but, remarkably, not for all graphs; see \cite{DH} for an explicit counterexample. As far as we know, there are no optimal results on Lieb--Thirring constants on quantum graphs apart from \eqref{eq:dh}. We emphasize that the proof in \cite{DH} proceeds by showing $L_{\gamma,N}\leq (4\pi)^{-1/2}\, \Gamma(\gamma+1)/\Gamma(\gamma+3/2)$ directly, without comparing $L_{\gamma,N}$ to $L_\gamma$.

In this paper we shall do exactly the latter, namely, we find a comparison method to relate $L_{\gamma,N}$ to $L_\gamma$, without needing to know the explicit value of $L_\gamma$. This allows us to settle the problem completely for even $N$ as well as, under a symmetry assumption, for odd $N$. The following two theorems are our main results.

\begin{theorem}\label{th1}
\label{main}
Let $\gamma\geq 1/2$. If $N$ is even, then
$$
L_{\gamma,N} = L_\gamma
$$
and, if $N$ is odd, then
$$
L_{\gamma,N} \leq \frac{N+1}{N}\, L_\gamma \,,
$$
where $L_{\gamma,N}$ and $L_\gamma$ are the optimal constants in \eqref{eq:ltstar} and \eqref{eq:lt}, respectively.
\end{theorem}

\noindent
\emph{Remarks.} 
(1) For even $N$, this theorem together with \eqref{eq:ltknown} yields explicitly the optimal constant for $\gamma=1/2$ and $\gamma\geq 3/2$. This improves our earlier bound from \cite{DH} for $\gamma\geq 2$. We emphasize that none of the methods used to prove \eqref{eq:ltknown} seem to generalize in an obvious way to $\Gamma_N$.\\
(2) A variant of our proof shows that if $L_{\gamma,N_0}=L_\gamma$ for some odd $N_0$, then $L_{\gamma,N}=L_\gamma$ for all $N\geq N_0$; see Proposition \ref{mono}.\\
(3) For $N=1$, our bound states $L_{\gamma,1}\leq 2L_\gamma$. The proof of Lemma \ref{neumann} shows that this bound is optimal as long as the optimal potential for $L_\gamma$ has a single bound state. This holds, in particular, for $\gamma=1/2$.\\
(4) For odd $N\geq 3$ our bound uses the bound $L_{\gamma,1}\leq 2 L_\gamma$ for $N=1$. If the latter bound can be improved for some (large) $\gamma$, then also our bounds for arbitrary odd $N\geq 3$ improve automatically. 

\medskip

We call a function $V$ on $\Gamma_N$ \emph{radial} if the value of $V(x)$ depends only on the distance of $x$ from the vertex of $\Gamma_N$. Let us denote by $L_{\gamma,N}^{(\rad)}$ the optimal constant in \eqref{eq:ltstar} \emph{when restricted to radial functions $V$}.

\begin{theorem}\label{th2}
\label{main2}
Let $\gamma\geq 1/2$. For any $N\geq 2$,
$$
L_{\gamma,N}^{(\rad)} = L_\gamma \,,
$$
where $L_{\gamma,N}^{(\rad)}$ is the optimal constant in the radial version of \eqref{eq:ltstar} and $L_\gamma$ is the optimal constant in \eqref{eq:lt}.
\end{theorem}

We will prove Theorem \ref{main} in Section \ref{sec:main} and Theorem \ref{main2} in Section \ref{sec:main2}.

\subsection*{Acknowledgement}
The author is grateful to T. Weidl for drawing my attention to Lieb--Thirring inequalities on quantum graphs and helpful comments.


\section{Proof of Theorem \ref{main}}\label{sec:main}

We begin with the proof of Theorem \ref{main} for $N=1$. This is the following bound on the eigenvalues of a half-line Schr\"odinger operator with Neumann boundary conditions. More precisely, this operator is defined via the quadratic form $\int_{\R_+} (|\psi'|^2 + V|\psi|^2)\,dx$ in $L_2(\R_+)$ with form domain $H^1(\R_+)\cap L_2(\R_+,V_+\,dx)$.

\begin{lemma}\label{neumann}
Let $H^{(\mathrm{Neu})}=-\frac{d^2}{dx^2} + V$ in $L_2(\R_+)$ with Neumann boundary conditions. Then, for all $\gamma\geq 1/2$,
$$
\tr \left( H^{(\mathrm{Neu})}\right)_-^\gamma \leq 2 L_\gamma \int_{\R_+} V_-^{\gamma+1/2}\,dx \,.
$$
\end{lemma}

\begin{proof}
We extend $V$ to a symmetric function $\tilde V$ on $\R$ and obtain, by the same arguments as in the proof of Theorem \ref{main2} below,
$$
\tr \left( H^{(\text{Neu})}\right)_-^\gamma + \tr \left( H^{(\text{Dir})}\right)_-^\gamma = \tr \left( H^{\R}\right)_-^\gamma \,,
$$
where $H^{(\text{Dir})}$ is the same as $H^{(\text{Neu})}$ but with Dirichlet boundary conditions and $H^\R$ is the operator $-\frac{d^2}{dx^2}+\tilde V$ in $L_2(\R)$. The claimed bound follows from the inequalities $\tr \left( H^{(\text{Dir})}\right)_-^\gamma\geq 0$ and \eqref{eq:lt}, that is,
$$
\tr \left( H^{\R}\right)_-^\gamma \leq L_\gamma \int_\R \tilde V_-^{\gamma+1/2}\,dx = 2 L_\gamma \int_{\R_+} V_-^{\gamma+1/2}\,dx \,.
\qedhere
$$
\end{proof}

We now turn to star graphs $\Gamma_N$ with $N\geq 3$. Lower bounds on the eigenvalues can be obtained by decoupling the edges. If we would decouple all the edges, we would end up with $N$ half-line Schr\"odinger operators with Neumann boundary conditions. Applying Lemma \ref{neumann} to each of these operators we would obtain the bound $L_{\gamma,N}\leq 2L_\gamma$, which is not optimal. The idea in the following proof is to apply a more subtle decoupling.

\begin{proof}[Proof of Theorem \ref{main}]
\textit{Case $N$ even}. We write $N=2n$ and consider the quadratic form $h^{(\text{cut})}[\psi]$, given by the same expression as $h[\psi]$, but with form domain
$$
\left\{ \psi\in L_2(\Gamma_N):\ \forall \ 1\leq j\leq N: \psi_j \in H^1(\R_+) \ \text{and}\ \forall \ 1\leq j \leq n :\ \psi_j(0)=\psi_{j+N}(0) \right\} \,. 
$$
In other words, we decompose $\Gamma_N$ into $n$ copies of $\R$, namely, $e_1\cup e_{n+1},\ldots, e_n\cup e_N$, where $e_1,\ldots, e_N$ are the edges of $\Gamma_N$. Since the form domain of $h^{(\text{cut})}$ contains that of $h$, the corresponding operator $H^{(\text{cut})}$ satisfies $H^{(\text{cut})}\leq H$ in the sense of quadratic forms, and therefore
\begin{equation}
\label{eq:mainproof1}
\tr H_-^\gamma \leq \tr (H^{(\text{cut})})_-^\gamma
\end{equation}
for any $\gamma$. Since for the operator $H^{(\text{cut})}$ each edge is only connected to one other edge, we have
$$
H^{(\text{cut})} \sim \bigoplus_{i=1}^n H_i \,,
$$
where $H_i$ is the Schr\"odinger operator in $L_2(\R)$ with potential $\tilde V_i$ given for $t> 0$ by
$$
\tilde V_i(t) = V_i(t) \,,
\qquad
\tilde V_i(-t) = V_{n+i}(t) \,.
$$
(Here, $V_i$ and $V_{n+i}$ denote the restrictions of $V$ to the $i$-th and $n+i$-th edge.) Thus,
\begin{equation}
\label{eq:mainproof2}
\tr (H^{(\text{cut})})_-^\gamma = \sum_{i=1}^n \tr (H_i)_-^\gamma \,.
\end{equation}
Finally, if $\gamma\geq 1/2$, we can use the Lieb--Thirring inequality \eqref{eq:lt} to bound
\begin{equation}
\label{eq:mainproof3}
\tr (H_i)_-^\gamma \leq L_\gamma \int_\R (\tilde V_i)_-^{\gamma+1/2}\,dt \,.
\end{equation}
Combining \eqref{eq:mainproof1}, \eqref{eq:mainproof2} and \eqref{eq:mainproof3} we obtain for $\gamma\geq 1/2$
$$
\tr H_-^\gamma \leq L_\gamma \sum_{i=1}^n \int_\R (\tilde V_i)_-^{\gamma+1/2}\,dt = L_\gamma \int_{\Gamma_N} V_-^{\gamma+1/2}\,dx \,,
$$
as claimed.

\textit{Case $N$ odd.} We shall show that for $\gamma\geq 1/2$,
\begin{equation}
\label{eq:mainproof4}
\tr H_-^\gamma \leq L_\gamma \int_{\Gamma_N} V_-^{\gamma+1/2}\,dx + L_\gamma \int_{\R_+} (V_N)_-^{\gamma+1/2} \,dt \,.
\end{equation}
After relabelling the edges this yields
$$
\tr H_-^\gamma \leq L_\gamma \int_{\Gamma_N} V_-^{\gamma+1/2}\,dx + L_\gamma \int_{\R_+} (V_i)_-^{\gamma+1/2} \,dt
$$
for any $i=1,\ldots,N$, and summing this inequality over $i$, we obtain
$$
N \tr H_-^\gamma \leq (N+1) L_\gamma \int_{\Gamma_N} V_-^{\gamma+1/2}\,dx\,,
$$
which is the claimed inequality.

Thus, it remains to prove \eqref{eq:mainproof4}. This time we define a quadratic form $h^{(\text{cut})}$ by the same expression as $h[\psi]$ but with form domain
$$
\left\{ \psi\in L_2(\Gamma_N):\ \forall \ 1\leq j\leq N: \psi_j \in H^1(\R_+) \ \text{and}\ \psi_1(0)=\ldots=\psi_{N-1}(0) \right\} \,.
$$
As before, we have \eqref{eq:mainproof1}. Since the $N$-th edge is disconnected from the rest of the edges, we have
$$
H^{(\text{cut})} \sim \tilde H \oplus \tilde H_N \,,
$$
where $\tilde H$ is the operator in $L_2(\Gamma_{N-1})$, which is obtained by ignoring the $N$-th edge, and $\tilde H_N$ is the Schr\"odinger operator in $L_2(\R_+)$ with potential $V_N$ and a Neumann boundary condition. Thus,
\begin{equation}
\label{eq:mainproof5}
\tr (H^{(\text{cut})})_-^\gamma = \tr \tilde H_-^\gamma + \tr (\tilde H_N)_-^\gamma \,.
\end{equation}
Since $N-1$ is even, we have according to Step 1
\begin{equation}
\label{eq:mainproof6}
\tr \tilde H_-^\gamma \leq L_\gamma \sum_{i=1}^{N-1} \int_{\R_+} (V_i)_-^{\gamma+1/2}\,dt \,.
\end{equation}
On the other hand, by Lemma \ref{neumann},
\begin{equation}
\label{eq:mainproof7}
\tr (\tilde H_N)_-^\gamma \leq 2 L_\gamma \int_{\R_+} (V_N)_-^{\gamma+1/2}\,dt \,.
\end{equation}
The claimed inequality \eqref{eq:mainproof4} now follows from \eqref{eq:mainproof1}, \eqref{eq:mainproof5}, \eqref{eq:mainproof6} and \eqref{eq:mainproof7}. This concludes the proof of the theorem.
\end{proof}

A refinement of the previous proof yields

\begin{proposition}\label{mono}
Let $\gamma\geq 1/2$. If $N_0<N$ are both odd, then
$$
L_{\gamma,N} \leq ((N-N_0)/N) L_\gamma + (N_0/N) L_{\gamma,N_0} \,.
$$
In particular, if $L_{\gamma,N_0}=L_\gamma$ for some odd $N_0\in\N$, then $L_{\gamma,N}=L_\gamma$ for all $N\geq N_0$.
\end{proposition}

Note that the bound in Theorem \ref{main} follows by taking $N_0=1$ and using $L_{\gamma,1}\leq 2L_\gamma$ according to Lemma \ref{neumann}.

\begin{proof}
We argue as in the odd $N$ case of Theorem \ref{main} and decouple $\Gamma_N$ into two star graphs $\Gamma_{N_0}$ and $\Gamma_{N-N_0}$. For $\Gamma_{N_0}$ we use the bound with $L_{\gamma,N_0}$ and for $\Gamma_{N-N_0}$ we use the bound with $L_\gamma$ (since $N-N_0$ is even). Finally, we sum over all possible choices of $N_0$ edges, as in the equations after \eqref{eq:mainproof4}. We omit the details.
\end{proof}


\section{Proof of Theorem \ref{main2}}\label{sec:main2}

We now turn our attention to radial potentials $V$ on $\Gamma_N$ and show that the constant $L_{\gamma,N}^{(\rad)}$ coincides with the optimal one-dimensional constant $L_{\gamma}$. This holds both for even and odd $N$.

The symmetry of $\Gamma_N$ allows one to construct an orthogonal decomposition of the space $L_2(\Gamma_N)$ which reduces the Kirchhoff Laplacian. If, in addition, $V$ is radial, it also reduces the  operator $H$. The study of the spectrum of $H$ is then reduced to the study of the spectrum of the orthogonal components in the decomposition, where each component can be identified with a Schr\"odinger operator acting in the space $L_2(\R_+).$ 

In \cite{Sololo, solnai, FK} a decomposition of the $L_2$ space was given for so-called regular, rooted metric trees. In what follows, we reformulate the decomposition of $L_2({\Gamma_N})$ for our purposes.
We denote by $\mathcal{H}^{(0)}$ the closed subspace of $L_2 (\Gamma_N)$ of all radial functions on $\Gamma_N$, i.e.,
$$
\mathcal{H}^{(0)}:= \{ \psi \in L_2(\Gamma_N): \forall r\geq 0 :\ \psi_1(r) =  \psi_2(r) = \ldots = \psi_N(r) \},
$$
where $\psi_j := \psi|_{e_j}$. Any radial function $\psi$ on $\Gamma_N$ can be identified with the function $s:= R \psi$ on the half-line $[0 , \infty)$ such that $\psi(x) = s(|x|)$ for each $ x \in  \Gamma_N,$ and

$$
\int_{\Gamma_N} | \psi (x)|^2 \,dx = N \int_0^\infty |s(x)|^2 \,dx,\quad \psi \in \mathcal{H}^{(0)}, s= R\psi.
$$
Thus, the operator $\sqrt{N} R$ defines an isometry of the subspace $\mathcal{H}^{(0)}$ onto the space $L_2(\R_+)$.

To state the orthogonal decomposition of $L_2(\Gamma_N)$ we define for $1 \leq \ell \leq N-1$, the following orthogonal subspaces,
\begin{eqnarray}\nonumber
\mathcal{H}^{(\ell)}&:=& \{ \psi \in L_2(\Gamma_N): \forall j=1,\ldots,N,\ \forall r\geq 0:\ \psi_{j+1}(r) = e^{i2\pi(\ell /N)} \psi_{j}(r) \}.
\end{eqnarray}
(Here, we write $\psi_{N+1}=\psi_1$.) Clearly, as for $\ell=0$ there are isometries from $H^{(\ell)}$ onto $L_2(\R_+)$.

\begin{lemma}
The subspaces $\mathcal{H}^{(\ell)},\ \ell=0, \ldots, N-1$, are mutually orthogonal and 
\begin{equation}\label{decomp}
L_2(\Gamma_N) = \bigoplus_{\ell=0}^{N-1} \mathcal{H}^{(\ell)}.
\end{equation}
\end{lemma}

\begin{proof}
First, we show that $ L_2(\Gamma_N) = \mbox{span} \ \{\mathcal{H}^{(\ell)}: \ell \}$, i.e., for every function $\psi \in L_2(\Gamma_N)$ there are functions $\psi^{(\ell)} \in \mathcal{H}^{(\ell)}$ such that $\psi= \sum_{\ell=0}^{N-1} \psi^{(\ell)}$. (Note that for $N=2$ this corresponds to the fact that every function on the real line is given as a sum of even and odd functions.)

We can write $\psi= \sum_{\ell=0}^{N-1} \psi^{(\ell)}$, where the functions $\psi^{(\ell)}$ are defined via their restrictions $\psi^{(\ell)}_k$ to the $k$-th edge, $k=1,\ldots,N$, by
$$
\psi_k^{(0)}(t) = \frac 1N \sum_{j=1}^N \psi_j(t)
$$
and, for $\ell=1,\ldots, N-1$,
$$
\psi_k^{(\ell)} = \frac 1N \left(\psi_k(t) +  \sum_{j \neq k} e^{i2\pi \ell/N} \psi_j(t) \right) \,.
$$
The identity $\psi= \sum_{\ell=0}^{N-1} \psi^{(\ell)}$ follows from the fact that
 $$
 \sum_{\ell =0}^{N-1} e^{i2 \pi \ell/N} =  \sum_{\ell =0}^{N-1} \left(e^{i2 \pi/N} \right)^{\ell}
 =\frac{\left(e^{i2 \pi/N}\right)^N-1}{ \left(e^{i2 \pi/N}\right) -1}=   0 \,.
 $$
Moreover, it is easy to verify that $\psi^{(\ell)}\in\mathcal H^{(\ell)}$.

To prove the lemma, it remains to show that the spaces $\mathcal{H}^{(\ell)}, \ 0 \leq \ell \leq N-1,$ are mutually orthogonal. For $\psi^{(\ell)} \in \mathcal{H}^{(\ell)}$ and  $\psi^{(m)} \in \mathcal{H}^{(m)}$ with $\ell \neq m$ consider
\begin{align*}
\int_{\Gamma} \psi^{(\ell)} \overline{\psi^{(m)}} \,dx 
& = \sum_{j=1}^N \int_{\R_+} \psi_j^{(\ell)} \overline{\psi_j^{(m)}} \,dt = \sum_{j=1}^N  \int_{\R_+} e^{2i \pi \ell (j-1)/N} \psi_1^{(\ell)} e^{-2i \pi m (j-1)/N} \overline{\psi_1^{(m)}}
\,dt \\
&=  \int_{\R_+}  \psi_1^{(\ell)} \overline{\psi_1^{(m)}} \,dt\ \sum_{j=1}^N \left( e^{2i\pi(\ell-m)/N} \right)^{j-1}.
\end{align*}
The right-hand side equals zero since
$$
\sum_{j=1}^N \left( e^{2i\pi(\ell-m)/N} \right)^{j-1} = \sum_{j=0}^{N-1} \left( e^{2i\pi(\ell-m)/N} \right)^j = \frac{\left(e^{i2 \pi (\ell -m) / N}\right)^N-1}{ \left(e^{i2 \pi (\ell -m) /N}\right) -1}=   0.
$$
Hence, the spaces $\mathcal{H}^{(\ell)}, \ 0 \leq \ell \leq N-1,$ are mutually orthogonal, as claimed.
\end{proof}

A function in $\mathcal H^{(\ell)}$ is completely determined by its restriction to one of the edges. We now characterize the $H^1(\Gamma_N)$ property of a function in $\mathcal H^{(\ell)}$ in terms of its restrictions. Clearly, a function in $\mathcal H^{(0)}$ belongs to $H^1(\Gamma_N)$ iff its restrictions belong to $H^1(\R_+)$. On the other hand, a function $\psi\in \mathcal{H}^{(\ell)}$ with $\ell=1,\ldots,N-1$ belongs to $H^1(\Gamma_N)$ iff its restrictions belong to $H^{1,0}(\R_+) = \{ \psi \in H^1(\R_+) : \psi(0) = 0 \}$. The crucial point here is the Dirichlet boundary condition at the origin. Moreover, we have
$$
\int_{\Gamma_N} |\psi'|^2\,dx = \sum_{\ell=0}^{N-1} \int_{\Gamma_N} |(\psi^{(\ell)})'|^2\,dx \,, 
$$
where $\psi^{(\ell)}$ denotes the projection of $\psi$ onto $\mathcal H^{(\ell)}$.

We conclude that the subspaces $\mathcal H^{(l)}$ reduce the Schr\"odinger operator $H$ and that the operators $H |_{ \mathcal{H}^{(\ell)}}$ are unitarily equivalent to operators $H^{(\ell)}$ in $L_2(\R_+)$. These operators act as $-\frac{d^2}{dx^2} + V(x)$ and have Neumann (if $\ell=0$) and Dirichlet (if $\ell=1,\ldots,N-1$) boundary conditions. Here we identify the radial function $V$ on $\Gamma_N$ in a natural way with a function on $\R_+$. (More precisely, the operators $H^{(\ell)}$ are defined via the quadratic form $\int_{\R_+} (|\psi'|^2 + V|\psi|^2 )\,dx$ with form domain $H^1(\R_+)$ for $\ell=0$ and $H^{1,0}(\R_+)$ for $\ell=1,\ldots, N-1$.) Clearly, the operators $H^{(\ell)}$ for $\ell=1,\ldots,N-1$ coincide.

To summarize, the operator $H$ in $L_2(\Gamma_N)$ is unitary equivalent to the orthogonal sum of the operators $H^{(\ell)}$ on $ L_2(\R_+)$,
\begin{equation}\label{druck}
H \sim  \bigoplus_{\ell=0}^{N-1} H^{(\ell)} \,,
\end{equation}
and therefore its eigenvalues, counting multiplicities, are given by the union of the eigenvalues of $H^{(\ell)}$, counting multiplicities. Then, for any $\gamma$,
\begin{equation}
\label{eq:evsumradial0}
\tr H_-^\gamma = \tr \left( H^{(0)} \right)_-^\gamma + (N-1) \tr \left( H^{(1)} \right)_-^\gamma \,.
\end{equation}

Consider now the Schr\"odinger operator
$$
\tilde H = -\frac{d^2}{dx^2} + \tilde{V}
\qquad\text{in}\ L_2(\R) \,,
$$
where the potential $\tilde{V}$ is the symmetric extension of the potential $V|_{e_j}$ to the whole-line. The unitary equivalence \eqref{druck} with $N=2$ implies that $\tilde H \sim H^{(0)} \oplus H^{(1)}$. Reinserting this into \eqref{druck} we find
$$
H \sim \tilde H \oplus  \bigoplus_{\ell=2}^{N-1} H^{(\ell)},
$$
and hence
\begin{equation}
\label{eq:evsumradial}
\tr H_-^\gamma = \tr \left( \tilde H \right)_-^\gamma + (N-2) \tr \left( H^{(1)} \right)_-^\gamma \,.
\end{equation}
This is the key identity in the radial case.

According to the Lieb--Thirring inequality \eqref{eq:lt}, for the first trace on the right side of \eqref{eq:evsumradial} and $\gamma\geq 1/2$ we have
$$
\tr \left( \tilde H \right)_-^\gamma \leq
L_\gamma \int_\R (\tilde V)_-^{\gamma+1/2} \,dx =
2 L_\gamma \int_{\R_+} V_-^{\gamma+1/2} \,dx \,.
$$
On the other hand, by the variational principle, inequality \eqref{eq:lt} remains also true for the eigenvalues of Dirichlet half-line operators, and therefore for the second trace on the right side of \eqref{eq:evsumradial} we have
$$
\tr \left( H^{(1)} \right)_-^\gamma \leq L_\gamma \int_{\R_+} V_-^{\gamma+1/2} \,dx \,.
$$
Thus, the right side of \eqref{eq:evsumradial} is bounded from above by
$$
N L_\gamma \int_{\R_+} V_-^{\gamma+1/2} \,dx = L_\gamma \int_{\Gamma_N} V_-^{\gamma+1/2}\,dx \,,
$$
which proves the bound $L^{(\rad)}_{\gamma,N}\leq L_\gamma$ claimed in Theorem 2.

Conversely, for any $\epsilon>0$ there is a compactly supported $V$ on $\R$ such that
\begin{equation}
\label{eq:ltopt}
\tr \left( -\frac{d^2}{dx^2} + V \right)_-^\gamma \geq (1-\epsilon) L_{\gamma} \int_{\R} V_-^{\gamma+1/2}\,dx \,.
\end{equation}
We denote by $V_a(x)=V(x-a)$ the translate of this potential and choose $a$ so large that the support of $V_a$ is contained in $\R_+$. We use $V_a$ as a radial potential on $\Gamma_N$ and denote the corresponding operator by $H_a$ and its parts on $\mathcal H^{(0)}$ and $\mathcal H^{(1)}$ by $H^{(0)}_a$ and $H^{(1)}_a$, respectively. It is easy to see that as $a\to\infty$,
$$
\frac{\tr\left( H^{(0)}_a \right)_-^\gamma}{\tr \left( -\frac{d^2}{dx^2} + V_a \right)_-^\gamma} \to 1
\qquad\text{and}\qquad
\frac{\tr\left( H^{(1)}_a \right)_-^\gamma}{\tr \left( -\frac{d^2}{dx^2} + V_a \right)_-^\gamma} \to 1 \,.
$$
On the other hand, by translation invariance, $\tr \left( -\frac{d^2}{dx^2} + V_a \right)_-^\gamma= \tr \left( -\frac{d^2}{dx^2} + V \right)_-^\gamma$ and $\int_{\Gamma_N} (V_a)_-^{\gamma+1/2}\,dx = N \int_{\R} V_-^{\gamma+1/2}\,dx$. Therefore \eqref{eq:evsumradial0} and \eqref{eq:ltopt} yield
$$
\liminf_{a\to\infty} \frac{\tr (H_a)_-^\gamma}{\int_{\Gamma_N} (V_a)_-^{\gamma+1/2}\,dx} \geq (1-\epsilon) L_\gamma \,.
$$
This proves $L^{(\rad)}_{\gamma,N}\geq (1-\epsilon) L_\gamma$ and, since $\epsilon>0$ is arbitrary, we obtain $L^{(\rad)}_{\gamma,N}\geq L_\gamma$. This concludes the proof of the theorem.
\qed

\bibliographystyle{plain}
\bibliography{diss.bib}

\end{document}